\documentclass[12pt,leqno]{article}
\usepackage{amsthm,amsmath,amssymb,enumerate}
\numberwithin{equation}{section}

\newtheorem{thm}{Theorem}[section]

\newtheorem{lem}[thm]{Lemma}

\theoremstyle{definition}
\newtheorem{rem}[thm]{Remark}

\newtheorem{ex}[thm]{Example}

\newcommand{\Aut}{\operatorname{Aut}}

\begin{document}
\title{Non-symmetric class $2$ association schemes obtained
  by doubling of skew-Hadamard matrices are non-schurian%
  \footnote{
    Ilia Ponomarenko pointed out that the main result of this article is already known,
    for example, Theorem 2.6.6 in [I. A. Faradžev, M. H. Klin, and M. E. Muzichuk,
    Cellular rings and groups of automorphisms of graphs,
    Investigations in algebraic theory of combinatorial objects, Math. Appl. (Soviet Ser.),
    vol. 84, Kluwer Acad. Publ., Dordrecht, 1994].
    However, he also said that the isomorphism of automorphism groups seems to be new.
  }}
\author{Akihide Hanaki}
\date{Department of Mathematics, Faculty of Science,\\
  Shinshu University, Matsumoto 390-8621, Japan\\
  e-mail : hanaki@shinshu-u.ac.jp}
\maketitle

% \keywords{association scheme; skew-Hadamard matrix}
% \subjclass[2010]{05E30}

\begin{abstract}
  We can obtain a non-symmetric class $2$ association scheme by a skew-Hadamard matrix.
  We begin with a skew-Hadamard matrix of order $n$, construct
  a skew-Hadamard matrix of order $2n$ by doubling construction, and
  a non-symmetric class $2$ association scheme of order $2n-1$.
  We will show that the association scheme obtained in this way never be schurian
  if $n$ is greater than or equal to $8$.
\end{abstract}

\section{Introduction}
Typical examples of association schemes are obtained by transitive permutation groups,
and such schemes are said to be schurian.
There are so many non-schurian association schemes,
but not so many constructions of non-schurian association schemes are known,
for example in \cite{VanDam-Koolen2005}, \cite[Theorem 26.4]{Wielandt}.
In this article, we will show that association schemes obtained by doubling of
skew-Hadamard matrices are non-schurian.

A matrix $H$ is said to be a \emph{Hadamard matrix} of order $n$
if $H$ an $n\times n$ matrix with entries in $\{\pm1\}$
such that $HH^T=nI_n$, where $H^T$ is the transpose of $H$ and $I_n$ is the identity matrix.
It is well known that, if a Hadamard matrix of order $n$ exists, then
$n=1,2$ or $n\equiv 0\pmod{4}$.
A Hadamard matrix $H$ is said to be \emph{skew} if $S=H-I_n$ is skew-symmetric,
namely $S^T=-S$.

Let $H$ be a skew-Hadamard matrix of order $n$.
By changing signs, we can assume that
$$H=\left[
  \begin{array}{c|ccc}
    1&1&\dots&1\\
    \hline
    -1&&&\\
    \vdots&&I_{n-1}+A_1-A_2&\\
    -1 &&&
  \end{array}
\right],$$
where $A_1, A_2$ are matrices with entries in $\{0,1\}$ and $A_1^T=A_2$.
It is known that $\mathfrak{X}=\{A_0=I_{n-1}, A_1, A_2\}$ is 
a non-symmetric class $2$ association scheme of order $n-1$.
Set $S=H-I_n$ and
$$H'=\left[
  \begin{array}{cc}
    I_n+S & I_n+S\\
    -I_n+S & I_n-S
  \end{array}
\right].$$
This is also a skew-Hadamard matrix and called the \emph{doubling} of $H$
\cite[Section 2.6]{Koukouvinos-Stylianou2008}.
By definition, we can write
$$H'=\left[
  \begin{array}{c|ccc}
    1&1&\dots&1\\
    \hline
    -1&&&\\
    \vdots&&I_{2n-1}+B_1-B_2&\\
    -1 &&&
  \end{array}
\right],$$
where $B_1, B_2$ are matrices with entries in $\{0,1\}$ and $B_1^T=B_2$.
In this case, $\mathfrak{Y}=\{B_0=I_{2n-1}, B_1, B_2\}$ is also
a non-symmetric class $2$ association scheme of order $2n-1$.
By construction, we can see that
$$B_1=\left[
  \begin{array}{ccc|c|ccc}
    &&&0&&&\\
    &A_1&&\vdots&&A_0+A_1&\\
    &&&0&&&\\
    \hline
    1&\dots&1&0&0&\dots&0\\
    \hline
    &&&1&&&\\
    &A_1&&\vdots&&A_2&\\
    &&&1&&&
  \end{array}
\right].$$
The main results of this article show that, if $n\geq 8$,
then the automorphism groups of $\mathfrak{X}$ and $\mathfrak{Y}$ are isomorphic,
the automorphism group of $\mathfrak{Y}$ is intransitive,
and thus $\mathfrak{Y}$ is non-schurian.

\section{Preliminary}
We define association schemes in matrix form.
Let $m$ be a positive integer.
An \emph{association scheme} of order $m$ is a set $\mathfrak{X}=\{A_0,\dots,A_d\}$ of
$m\times m$ matrices with entries in $\{0,1\}$ such that
(1) $A_0=I_m$, the identity matrix,
(2) $\sum_{i=0}^d A_i=J_m$, the matrix with all entries $1$, 
(3) for $i \in \{0,1,\dots,d\}$, $A_i^T\in \mathfrak{X}$, and
(4) for all $0\leq i, j\leq d$, the product $A_iA_j$ is a linear combination of $\mathfrak{X}$.
The number $d$ is called the \emph{class} of $\mathfrak{X}$.
An association scheme $\mathfrak{X}$ is said to be \emph{symmetric} if $A_i^T=A_i$
for all $1\leq i\leq d$.
An association scheme $\mathfrak{X}$ is said to be \emph{commutative} if
$A_iA_j=A_jA_i$ for all $0\leq i, j\leq d$.
It is known that all association schemes of class at most $4$ are commutative
\cite[Theorem 4.5.1 (ii)]{Zi}.

Let $\mathfrak{X}=\{A_0,\dots,A_d\}$ be an association scheme of order $m$.
For a permutation $\sigma$ of degree $m$, we denote by $P_\sigma$ the permutation matrix 
corresponding to $\sigma$.
The \emph{automorphism group} $\Aut(\mathfrak{X})$ of $\mathfrak{X}$
is defined by
$$\Aut(\mathfrak{X})=\{\sigma\in S_m\mid \text{$P_\sigma^TA_iP_\sigma=A_i$ for all $0\leq i\leq d$}\}.$$

Let $G$ be a transitive permutation group on $X=\{1,\dots,m\}$.
We define the action of $G$ on $X\times X$ by $(x,y)^g=(x^g,y^g)$.
Let $X\times X=R_0\cup\dots\cup R_d$ be the orbit partition of this action,
where $R_0=\{(x,x)\mid x\in X\}$.
Let $A_i$ be the adjacency matrix of $R_i$.
Then $\mathfrak{X}=\{A_0,\dots,A_d\}$ is an association scheme.
An association scheme obtained by a transitive permutation group is said to be \emph{schurian}.
If $\mathfrak{X}$ is schurian, then $\Aut(\mathfrak{X})$ is transitive and 
$\Aut(\mathfrak{X})$ defines $\mathfrak{X}$.

\begin{ex}[Non-symmetric class $2$ association scheme]
  Let $\mathfrak{X}=\{A_0,A_1,A_2\}$ be a non-symmetric class $2$ association scheme of order $m$.
  Such association scheme exists only if $m\equiv 3\pmod{4}$.
  Since $\mathfrak{X}$ is non-symmetric, $A_1^T=A_2$.
  The products of matrices are :
  \begin{eqnarray*}
  A_1^2 &=& \frac{m-3}{4}A_1+\frac{m+1}{4}A_2,\quad
            A_2^2 =\frac{m+1}{4}A_1+\frac{m-3}{4}A_2,\\
  A_1A_2 &=&A_2A_1 = \frac{m-1}{2}A_0+\frac{m-3}{4}A_1+\frac{m-3}{4}A_2.
  \end{eqnarray*}
  Eigenvalues of $A_1$ are $(m-1)/2$ and $(-1\pm\sqrt{-m})/2$,
  and thus $A_1$ is invertible.
  In this case, $\sigma\in S_m$ is an automorphism of $\mathfrak{X}$ if and only if
  $P_\sigma^TA_1P_\sigma=A_1$.
\end{ex}

\section{Results}
In the rest of this article, as in Introduction,
$H$ is a skew-Hadamard matrix of order $n$, $H'$ is the doubling of $H$,
and $\mathfrak{X}=\{A_0,A_1,A_2\}$ and $\mathfrak{Y}=\{B_0,B_1,B_2\}$
are non-symmetric class $2$ association schemes obtained by $H$ and $H'$, respectively.
We assume that $n\geq 8$.
Recall that
\begin{equation}\label{eq1}
B_1=\left[
  \begin{array}{ccc|c|ccc}
    &&&0&&&\\
    &A_1&&\vdots&&A_0+A_1&\\
    &&&0&&&\\
    \hline
    1&\dots&1&0&0&\dots&0\\
    \hline
    &&&1&&&\\
    &A_1&&\vdots&&A_2&\\
    &&&1&&&
  \end{array}
\right],
\end{equation}
\begin{eqnarray}
  A_1^2 &=& \frac{n-4}{4}A_1+\frac{n}{4}A_2,\quad
            A_2^2 =\frac{n}{4}A_1+\frac{n-4}{4}A_2,\label{eq2}\\
  A_1A_2 &=&A_2A_1 = \frac{n-2}{2}A_0+\frac{n-4}{4}A_1+\frac{n-4}{4}A_2,\label{eq3}
\end{eqnarray}
and
\begin{eqnarray}
  B_1^2 &=& \frac{n-2}{2}B_1+\frac{n}{2}B_2,\quad
            B_2^2 = \frac{n}{2}B_1+\frac{n-2}{2}B_2.\label{eq4}\\
  B_1B_2 &=&B_2B_1 = (n-1)B_0+\frac{n-2}{2}B_1+\frac{n-2}{2}B_2.\label{eq5}
\end{eqnarray}

For a matrix $M$, we denote by $M_{ij}$ the $(i,j)$-entry of $M$.
%and by $M_i$ the $i$-th row of $M$.
For $i \in \{1,\dots,2n-1\}$, we set
$N(i)=\{j \mid 1\leq j\leq 2n-1,\ (B_1)_{ij}=1\}$.
For $1\leq i<j\leq 2n-1$, we set
$N(i,j)=N(i)\cap N(j)$,
and for $1\leq i<j<k\leq 2n-1$, we set
$N(i,j,k)=N(i)\cap N(j)\cap N(k)$.
Also we set %$\nu(i,j)=|N(i,j)|$ and
$\nu(i,j,k)=|N(i,j,k)|$.
The number $\nu(i,j,k)$ was used in \cite{Hanaki-Kharaghani-Mohammadian-TayfehRezaie2020}.
We remark that $|N(i,j)|=(B_1B_1^T)_{ij}=(B_1B_2)_{ij}=(n-2)/2$ by (\ref{eq5}).
We also remark that $\nu(i,j,k)=\nu(i^\sigma,j^\sigma,k^\sigma)$ for $\sigma\in\Aut(\mathfrak{Y})$.

The next lemma is crucial for the proof of our mail results.

\begin{lem}\label{lem1}
  For $1\leq i<j<k\leq 2n-1$, $\nu(i,j,k)\leq (n-2)/2$ holds.
  The equality holds if and only if $(i,j,k)=(a,n,n+a)$ for some $1\leq a\leq n-1$.
\end{lem}

\begin{proof}
  Since $|N(i,j)|=(n-2)/2$, the inequality $\nu(i,j,k)\leq (n-2)/2$ is clear
  and the equality holds if and only if
  $N(i,j,k)=N(i,j)=N(i,k)=N(j,k)$.

  Case 1.
  First, we consider the case $(i,j,k)=(a,n,n+a)$ for some $1\leq a\leq n-1$.
  In this case, by (\ref{eq1}), $N(n)=\{1,\dots,n-1\}$ and $N(a,n)=N(n,n+a)$.
  Thus the equality holds.

  Case 2.
  Suppose $n\in\{i,j,k\}$ and not in Case 1.
  By (\ref{eq1}) and (\ref{eq3}), $\nu(i,j,k)=(n-4)/4$
  and $\nu(i,j,k)\ne (n-2)/2$.

  Case 3.
  Suppose $i<j<k<n$.
  We assume $\nu(i,j,k)=(n-2)/2$.
  Thus $N(i,j,k)=N(i,j)=N(i,k)=N(j,k)$.
  By (\ref{eq3}), we can set $N(i,j)\cap\{1,\dots,n-1\}=\{a_1,\dots,a_{(n-4)/4}\}$
  and $N(i,j)\cap\{n+1,\dots,2n-1\}=\{a_1+n,\dots,a_{(n-4)/4}+n, \ell\}$ for $\ell=i+n$ or $j+n$.
  Here $\ell$ comes from $A_0$.
  Thus $|\{i+n,j+n\}\cap  N(i,j,k)|=1$.
  Similarly $|\{j+n,k+n\}\cap  N(i,j,k)|=1$ and $|\{i+n,k+n\}\cap  N(i,j,k)|=1$.
  This is impossible.

  Case 4.
  Suppose $i<j<n<k$.
  We assume $\nu(i,j,k)=(n-2)/2$.
  Thus $N(i,j,k)=N(i,j)=N(i,k)=N(j,k)$.
  Similar to Case 3, we can set $N(i,j)\cap\{1,\dots,n-1\}=\{a_1,\dots,a_{(n-4)/4}\}$
  and $N(i,j)\cap\{n+1,\dots,2n-1\}=\{a_1+n,\dots,a_{(n-4)/4}+n, \ell\}$ for $\ell=i+n$ or $j+n$.
  Since the sets of positions of ones in the $(k-n)$-th rows of $A_1$ and $A_2$ are disjoint,
  $a_{m+n}$ is not in $N(k)$ for $1\leq m\leq n-1$.
  Thus $n/4=|N(i,j)\cap\{n+1,\dots,2n-1\}|=1$.
  This contradicts to the assumption $n\geq 8$.

  Case 5.
  Suppose $i<n<j<k$.
  We assume $\nu(i,j,k)=(n-2)/2$.
  Thus $N(i,j,k)=N(i,j)=N(i,k)=N(j,k)$.
  In this case, $n\in N(j,k)$ but $n$ is not in $N(i)$.
  This is a contradiction.

  Case 6.
  Suppose $n<i<j<k$.
  We assume $\nu(i,j,k)=(n-2)/2$.
  Thus $N(i,j,k)=N(i,j)=N(i,k)=N(j,k)$.
  By (\ref{eq3}) and $n\geq 8$, $|N(i,j)\cap\{1,\dots,n-1\}|=|N(i,j)\cap\{n+1,\dots,2n-1\}|=(n-4)/4>0$.
  We set $M(t)=(N(t)\setminus N(i,j,k))\cap\{1,\dots,n-1\}$ for $t\in \{i,j,k\}$.
  We can see that $(N(i,j,k)\cap\{1,\dots,n-1\})\cup M(i)\cup M(j)\cup M(k)$ is a disjoint union
  and $|M(i)|=n/4$.
  By counting elements, we have 
  $(N(i,j,k)\cap\{1,\dots,n-1\})\cup M(i)\cup M(j)\cup M(k)=\{1,\dots,n-1\}$.
  Suppose $a+n\in N(i,j,k)\cap\{n+1,\dots,2n-1\}$ ($1\leq a\leq n-1$).
  Since the sets of positions of ones in the $(t-n)$-th rows ($t=i,j,k$) of $A_1$ and $A_2$ are disjoint,
  $a\notin (N(i,j,k)\cap\{1,\dots,n-1\})\cup M(i)\cup M(j)\cup M(k)=\{1,\dots,n-1\}$.
  This is a contradiction.
\end{proof}

\begin{lem}\label{lem2}
  The sets $\{1,\dots,n-1\}$, $\{n\}$, $\{n+1,\dots,2n-1\}$ are closed by the action of
  the automorphism group $\Aut(\mathfrak{Y})$.
\end{lem}

\begin{proof}
  Let $\sigma$ be an automorphism of $\mathfrak{Y}$.
  Since $\nu(i,j,k)=\nu(i^\sigma,j^\sigma,k^\sigma)$,
  Lemme \ref{lem3} shows that $n^\sigma=n$.
  Now $\{1,\dots,n-1\}$ and $\{n+1,\dots,2n-1\}$ are fixed by $\sigma$
  by the form of $n$-th row of $B_1$ in (\ref{eq1}).
\end{proof}

\begin{lem}\label{lem3}
  The automorphism group $\Aut(\mathfrak{Y})$ is isomorphic to $\Aut(\mathfrak{X})$.
\end{lem}

\begin{proof}
  Suppose $\sigma\in \Aut(\mathfrak{Y})$.
  By Lemma \ref{lem2}, we can write
  $$P_{\sigma}=\left[
    \begin{array}{ccc}
      P_\tau&& \\
       &1 &\\
       &&P_\mu
    \end{array}
  \right]$$
  for some permutation matrices $P_\tau$ and $P_\mu$ of degree $n-1$.
  Now $B_1=P_\sigma^T B_1 P_\sigma$ shows
  $$\left[
    \begin{array}{ccc|c|ccc}
      &&&0&&&\\
      &A_1&&\vdots&&A_0+A_1&\\
      &&&0&&&\\
      \hline
      1&\dots&1&0&0&\dots&0\\
      \hline
      &&&1&&&\\
      &A_1&&\vdots&&A_2&\\
      &&&1&&&
    \end{array}
  \right]
  =\left[
    \begin{array}{ccc|c|ccc}
      &&&0&&&\\
      &P_\tau^TA_1P_\tau&&\vdots&&P_\tau^T(A_0+A_1)P_\mu&\\
      &&&0&&&\\
      \hline
      1&\dots&1&0&0&\dots&0\\
      \hline
      &&&1&&&\\
      &P_\mu^T A_1 P_\tau&&\vdots&&P_\mu^T A_2 P_\mu&\\
      &&&1&&&
    \end{array}
  \right].$$
  By $A_1=P_\tau^TA_1 P_\tau$, $P_\tau$ is an automorphism of $\mathfrak{X}$.
  Since $P_\tau$ and $A_1$ are invertible, $A_1=P_\mu^T A_1 P_\tau$ shows $\tau=\mu$.
  Thus $\sigma$ defines an automorphism $\tau$ of $\mathfrak{X}$.
  
  Conversely, it is clear that an automorphism of $\mathfrak{X}$ defines
  an automorphism of $\mathfrak{Y}$.
\end{proof}

Now we state our main result.
In general, one skew-Hadamard matrix gives non-isomorphic association schemes,
see \cite{Hanaki-Kharaghani-Mohammadian-TayfehRezaie2020}, for example.
Our argument is valid only for the construction in Introduction.
To clarify our result, we will not use skew-Hadamard matrices in the statement.

\begin{thm}\label{thm4}
  Let $\mathfrak{X}=\{A_0,A_1,A_2\}$  be a non-symmetric class $2$ association scheme
  of order $m\geq 7$.
  Set $B_0=I_{2m+1}$, $B_1$ by (\ref{eq1}) and $B_2=B_1^T$.
  Then $\mathfrak{Y}=\{B_0, B_1, B_2\}$ is a non-symmetric class $2$ association scheme
  of order $2m+1$,
  the automorphism group of $\mathfrak{Y}$ is intransitive,
  and thus $\mathfrak{Y}$ is non-schurian.
  Moreover, the automorphism group $\Aut(\mathfrak{Y})$ is isomorphic to $\Aut(\mathfrak{X})$.
\end{thm}

\begin{proof}
  The first statement is by Lemma \ref{lem2} and the second statement is by Lemma \ref{lem3}.
\end{proof}

\begin{rem}
  If we consider the case $n=4$, then we can get
  a non-symmetric class $2$ association scheme of order $7$.
  Such an association scheme is unique up to isomorphism,
  that is corresponding to the Fano plane and schurian.
\end{rem}

\section*{Acknowledgments}
This work was supported by JSPS KAKENHI Grant Number JP17K05165.

\bibliographystyle{amsplain}

\begin{thebibliography}{1}

\bibitem{Hanaki-Kharaghani-Mohammadian-TayfehRezaie2020}
A.~Hanaki, H.~Kharaghani, A.~Mohammadian, and B.~Tayfeh-Rezaie,
  \emph{Classification of skew‐{H}adamard matrices of order 32 and
  association schemes of order 31}, J. Combin. Des. \textbf{28} (2020), no.~6,
  421--427.

\bibitem{Koukouvinos-Stylianou2008}
C.~Koukouvinos and S.~Stylianou, \emph{On skew-{H}adamard matrices}, Discrete
  Math. \textbf{308} (2008), no.~13, 2723--2731. %\MR{2413970}

\bibitem{VanDam-Koolen2005}
E.~R. van Dam and J.~H. Koolen, \emph{A new family of distance-regular graphs
  with unbounded diameter}, Invent. Math. \textbf{162} (2005), no.~1, 189--193.

\bibitem{Wielandt}
H.~Wielandt, \emph{Finite permutation groups}, Translated from the German by R.
  Bercov, Academic Press, New York, 1964.

\bibitem{Zi}
P.-H. Zieschang, \emph{An algebraic approach to association schemes}, Lecture
  Notes in Mathematics, vol. 1628, Springer-Verlag, Berlin, 1996.

\end{thebibliography}
\providecommand{\href}[2]{#2}

\end{document}